\documentclass[11pt]{amsart}
\usepackage[margin=1in]{geometry}
\usepackage{xcolor}
\usepackage[english]{babel}

\newcommand{\vp}{\varphi}

\newcommand{\iqt}{\int_{Q_T}}
\newcommand{\wnr}{\|w\|_{r,Q_T}}
\newcommand{\lz}{L_N}

\newcommand{\mq}{M_q}

\def\XXint#1#2#3{{\setbox0=\hbox{$#1{#2#3}{\int}$ }
		\vcenter{\hbox{$#2#3$ }}\kern-.6\wd0}}

 \numberwithin{equation}{section}

\newcommand{\ra}{\rightarrow}
\newcommand{\ve}{\varepsilon}

\newcommand{\pt}{\partial_t}

\newcommand{\RN}{\mathbb{R}^N}

\newcommand{\irn}{{\int_{\RN}}}

\newcommand{\fr}{f(r)}
\newcommand{\fp}{f^{\prime}(r)}

\newcommand{\rn}{\mathbb{R}^N}

\newcommand{\irnt}{\int_{ Q_T}}

\newcommand{\vo}{v^{(0)}}
\newcommand{\uo}{u^{(0)}}
\newcommand{\bo}{b^{(0)}}
\newcommand{\pj}{\partial_{x_j}}
\newtheorem{theorem}{Theorem}[section]

\newtheorem{lemma}[theorem]{Lemma}
\newtheorem{clm}[theorem]{Claim}

\theoremstyle{definition}

\newtheorem{remark}{Remark}[section]

\title[magnetohydrodynamic equations with large initial data
] 
      { Global regularity for solutions of magnetohydrodynamic equations with large initial data}

\author[Xiangsheng Xu]{}
 \email{xxu@math.msstate.edu}

\begin{document}

\maketitle

\centerline{\scshape Xiangsheng Xu}
\medskip
{\footnotesize
 \centerline{Department of Mathematics \& Statistics}
   \centerline{Mississippi State University}
   \centerline{ Mississippi State, MS 39762, USA}
} 

	\begin{abstract}We study the regularity properties of solutions to the initial value problem for the magnetohydrodynamic equations in $\mathbb{R}^N, N\geq 3$. We obtain a global in-time strong solution without any smallness assumptions on the initial data. 

\end{abstract}
\bigskip


\section{Introduction}
 Magneto-hydrodynamics (MHD) couples Maxwell’s equations
of electromagnetism with hydrodynamics to describe the
macroscopic behavior of conducting fluids such as plasmas. It plays a very important role in solar physics, astrophysics, space
plasma physics, and in laboratory plasma experiments. The initial value problem for these equations reads
\begin{eqnarray}
	\pt u+(u\cdot\nabla) u-(b\cdot\nabla) b+\nabla p&=&\nu\Delta u\ \ \mbox{in $\rn\times(0,T)\equiv Q_T$},\label{nsf1}\\
	\pt b+(u\cdot\nabla) b-(b\cdot\nabla) u&=&\eta\Delta b\ \ \mbox{in $ Q_T$},\label{nsf2}\\
	\nabla\cdot u&=&	\nabla\cdot b=0\ \ \mbox{in $ Q_T$},\label{nsf4}\\
	u(x,0)&=& u^{(0)}(x),\ \ b(x,0)=b^{(0)}(x) \ \ \mbox{on $\rn$}\label{nsf5},
\end{eqnarray}
where $u$ denotes the velocity field, $b$ the magnetic field, $p$ the pressure, $\nu > 0$ the
kinematic viscosity, and $\eta >0$ the magnetic diffusivity. We refer the reader to \cite{DA} for a rather comprehensive survey on the subject of MHD. 

 The MHD equations offer an interesting interplay between the Navier-Stokes equations of hydrodynamics and electromagnetism. Mathematical analysis of their rich solution properties has always attracted a lot of attention. However, we will not attempt to offer a review of exiting results. For that we refer the reader to \cite{WZ}. 
One of the fundamental problems concerning the MHD equations is whether physically relevant regular solutions remain smooth for all time or they develop  singularities in finite
time. This problem is extremely difficult even in the 2D case \cite{WU1}. It has remained elusive in spite of many partial results. See \cite{DL,JJ,ZZ1,ZZ,Y} and the references therein.
The objective of this paper is to settle this open problem. More precisely, we have
\begin{theorem}\label{thm}Assume that
	\begin{eqnarray}
		|\uo| &\in &L^{\infty}(\rn)\cap  L^{2}(\rn)\ \ \mbox{with}\ \ \nabla\cdot\uo=0\ \mbox{and}\nonumber\\
	|\bo|&\in &L^{\infty}(\rn)\cap  L^{2}(\rn)\ \ \mbox{with $\nabla\cdot\bo=0$}.\nonumber
	\end{eqnarray}
	Let $(u, b)$ be a local-in-time strong solution to \eqref{nsf1}-\eqref{nsf5}. Define
	\begin{equation}
		w=|u|^2+|b|^2. 	\nonumber
	\end{equation}
	Then for each $\lz\geq 1$ there exist two positive numbers $c=c\left(N,\nu,\eta, \lz\right)$ and $s_5=s_5(N, \lz)$ such that
	\begin{eqnarray}\label{es}
		\|w\|_{\infty,Q_T}&\leq& 32\|w(\cdot,0)\|_{\infty,\rn}+c\|w\|_{\lz,Q_T}^{s_5}.
	\end{eqnarray}
\end{theorem}
The notion of a strong solution here  is adopted from \cite{OP}.  
It is understood to be a weak solution with the additional property
\begin{eqnarray}
w\in L^\infty(Q_T).\nonumber
\end{eqnarray}	
A bootstrap argument immediately yields high regularity for $(u, b)$. 
In fact, a strong solution can be shown to be a smooth one, thereby satisfying system \eqref{nsf1}-\eqref{nsf5} in the classical sense \cite{OP}. We will not pursue the details here. The local existence of a strong solution is largely known. For example, one can establish this by adopting the method in \cite{OP}. Later, we will see that there is a positive number
$c=c(N, \nu, \eta)$ such that
\begin{equation}\label{wi1}
	\|w\|_{\frac{N+2}{N}, Q_T}\leq c\left(\|\uo\|_{2,\rn}^2+\|\bo\|_{2,\rn}^2\right).
\end{equation}
This combined with Theorem \ref{thm} implies that a local-in-time strong solution never develops singularity. As a result, it can be extended as a global strong solution. If $b\equiv 0$, Theorem \ref{thm} implies a positive answer to the Navier-Stokes millennium problem \cite{F,X1}.

To describe our approach, let us first recall the classical regularity theorem for the initial value problem for linear parabolic equation of the form
\begin{eqnarray*}
	\pt v-\Delta v&=&\partial_{x_i}f_i(x,t)\ \ \mbox{in $Q_T$},\\
	v(x,0)&=&\vo\ \ \mbox{on $\rn$}.
\end{eqnarray*}
	Here we have employed the notation convention of summing over repeated indices. It asserts that for each $q>N+2$ there is a positive number $c=c(N,q)$ such that
\begin{eqnarray}\label{int1}
	\|v\|_{\infty, Q_T}\leq \|\vo\|_{\infty,\rn}+c\sum_{i=1}^{N}\|f_i\|_{q,Q_T}.
\end{eqnarray}
In the incompressible Navier-Stokes equations, which are \eqref{nsf1} and \eqref{nsf4} with $b\equiv0$, $f_i$ is roughly $v_i\sum_{j=1}^{N}v_j+p$. As a result, \eqref{int1} becomes
\begin{eqnarray}\label{int2}
	\|v\|_{\infty, Q_T}\leq \|\vo\|_{\infty,\rn}+c\left(\|v\|_{2q,Q_T}^2+	\|p\|_{q,Q_T}\right)\leq \|\vo\|_{\infty,\rn}+c\|v\|_{2q,Q_T}^2	.
\end{eqnarray}
The last step is due to \eqref{use1} below.
Thus, to be able to apply the classical theorem here, one must  have $v\in L^{2q}(Q_T)$ for some $q>N+2$. Of course, this integrability requirement can be weakened due to the work of Serrin, Prodi, and Ladyzenskaja \cite{OP}.  In spite of that, the gap between what is needed and what is available in terms of integrability is still substantial because we only have 
\begin{eqnarray*}
	v\in L^{\frac{2(N+2)}{N}}(Q_T)
\end{eqnarray*}
for a weak solution $v$ to the Navier-Stokes equations. On the other hand, for $\ell>q$ the interpolation inequality asserts
\begin{eqnarray*}
	\|v\|_{2q,Q_T}\leq \|v\|_{2\ell,Q_T}^{\frac{\ell(Nq-N-2)}{q(N\ell-N-2)}}\|v\|_{\frac{2(N+2)}{N},Q_T}^{\frac{(N+2)(\ell-q)}{q(N\ell-N-2)}}.
\end{eqnarray*} 
Insert this into \eqref{int2} to get
\begin{eqnarray}
		\|v\|_{\infty, Q_T}\leq \|\vo\|_{\infty,\rn}+c\|v\|_{2\ell,Q_T}^{\frac{2\ell(Nq-N-2)}{q(N\ell-N-2)}}\|v\|_{\frac{2(N+2)}{N},Q_T}^{\frac{2(N+2)(\ell-q)}{q(N\ell-N-2)}}	.\nonumber
\end{eqnarray}
This indicates that if we increase $q$  we can make the exponent of the corresponding $L^{2q}$ norm  decreases. Theorem \ref{thm} implies that we can also go in the opposite direction. That is, by increasing the power of the $L^{2q}$ norm  we can make $q$ smaller. 
How to achieve this is based upon three ingredients: scaling of the (modified) equations by a suitable $L^r$-norm, the interpolation inequality for the $L^r$-norms, and a carefully-designed iteration scheme of the  De Giorgi type. These techniques enable us to introduce several parameters. The trick is to find the right combination of these parameters. That is, instead of establishing a priori estimates through test functions, we are doing so here via suitable parameters. Our method seems to be very effective in dealing with the type of nonlinearity appearing in Navier-Stokes equations \cite{X1,X2}.   
 
Throughout this paper the following two inequalities will be used without acknowledgment:
\begin{eqnarray*}
	(|a|+|b|)^\gamma&\leq&\left\{\begin{array}{ll}
		2^{\gamma-1}(|a|^\gamma+|b|^\gamma)&\mbox{if $\gamma\geq 1$},\\
		|a|^\gamma+|b|^\gamma&\mbox{if $\gamma\leq 1$}.
	\end{array}\right.
\end{eqnarray*}
We also need Young's inequality 
\begin{equation}\label{you}
	|ab|\leq \ve|a|^p+\ve^{-\frac{q}{p}}|b|^q, 
\end{equation}
where $\ve, p, q\in(0, \infty)$ with $\frac{1}{p}+\frac{1}{q}=1$.

\begin{lemma}\label{ynb}
	Let $\{y_n\}, n=0,1,2,\cdots$, be a sequence of positive numbers satisfying the recursive inequalities
	\begin{equation*}
		y_{n+1}\leq cd^ny_n^{1+\alpha}\ \ \mbox{for some $d>1$ and $ c, \alpha\in (0,\infty)$.}
	\end{equation*}
	If
	\begin{equation*}
		y_0\leq c^{-\frac{1}{\alpha}}d^{-\frac{1}{\alpha^2}},
	\end{equation*}
	then $\lim_{n\rightarrow\infty}y_n=0$.
\end{lemma}
A De Giorgi iteration scheme relies upon this lemma, whose proof can be found in (\cite{D}, p.12).
The remainder of this paper is devoted to the proof of Theorem \ref{thm}.
\section{Proof of Theorem \ref{thm}}

As observed in \cite{WU1},  the
standard approach to the global regularity problem for the incompressible MHD
equations is to divide the process into two steps. The first step is to show the local
existence of a regular solution. In our case, this can be done via the contraction mapping principle
and its variants such as successive approximations \cite{OP}. 
The second step is to extend the local solution of the first step into a global (in time)
one by establishing suitable global a priori bounds on the solutions. That is, one
needs to obtain a priori estimates whose upper bounds are finite for each fixed $T$.
Once we have enough these so-called global bounds, the standard
Picard type extension theorem allows us to extend the local solution into a global
one. Therefore, the global regularity problem on the MHD equations boils down
to the global a priori bounds for local smooth solutions. Thus, in our subsequent calculations we can assume that our solutions are local-in-time strong ones. The key here is that our upper bounds are independent of $T$. This implies that our local-in-time strong solutions never develop singularity. As a result, they can be extended as global-in-time strong solutions. 

The free energy associated with \eqref{nsf1}-\eqref{nsf5} is well known. For completeness, we reproduce it in the following lemma.
\begin{lemma} There holds
	\begin{equation}
		\sup_{0\leq t\leq T}\irn\left(|u|^2+|b|^2\right)dx+\irnt\left(|\nabla u|^2+|\nabla b|^2\right)dxdt\leq c(\nu,\eta)\irn\left(|\uo|^2+|\bo|^2\right)dx.\label{fen}
	\end{equation}
\end{lemma}
\begin{proof} We easily verify that
	\begin{eqnarray}
		\pt u\cdot u&=&\frac{1}{2}\pt|u|^2,\label{am1}\\
		(u\cdot\nabla) u\cdot u&=&u_j\pj u_iu_i=\frac{1}{2}(u\cdot\nabla)|u|^2=\frac{1}{2}u\cdot\nabla|u|^2,\nonumber\\
		\Delta u\cdot u&=&\frac{1}{2}\Delta |u|^2-|\nabla u|^2.\nonumber
	\end{eqnarray}
 Equipped with \eqref{nsf4}, we further have that
	\begin{eqnarray}
		\irn b\cdot\nabla \xi dx&=&		\irn u\cdot\nabla \xi dx=0\ \ \mbox{for each $\xi\in W^{1,2}(\rn)$},\label{am6}\\
		(b\cdot\nabla) b\cdot u&=&b_j\pj b_iu_i=\nabla\cdot\left(b\otimes b u\right)-(b\cdot\nabla) u\cdot b,\label{am4}\\
		\nabla p\cdot u&=&\nabla\cdot (pu),\nonumber
	\end{eqnarray}
	where
	\begin{equation}
		b\otimes b=bb^T,\nonumber
	\end{equation}
	With these in mind, we use $u$ as a test function in \eqref{nsf1} to drive
	\begin{eqnarray}
		\frac{1}{2}\frac{d}{dt}\irn|u|^2dx+\nu\irn|\nabla u|^2dx=-\irn (b\cdot\nabla) u\cdot b\ dx.\label{ueq3}
	\end{eqnarray}
Similarly, by using $b$ as a test function in \eqref{nsf2}, we arrive at 
	\begin{eqnarray*}
		\frac{1}{2}\frac{d}{dt}\irn|b|^2dx+\eta\irn |\nabla b|^2dx=\irn (b\cdot\nabla) u\cdot b\ dx.
	\end{eqnarray*}
	Add this to \eqref{ueq3} to obtain
	\begin{eqnarray*}
		\frac{1}{2}\frac{d}{dt}\irn\left(|u|^2+|b|^2\right)dx+\nu\irn|\nabla u|^2dx+\eta\irn |\nabla b|^2dx=0.
	\end{eqnarray*}
	An integration with respect to $t$ yields the lemma.
\end{proof}

Recall that the Sobolev inequality in the whole space asserts
\begin{equation}
	\|f\|_{\frac{2N}{N-2},\rn}\leq c(N)\|\nabla f\|_{2,\rn}\ \ \mbox{for each $f\in H^1(\rn)$}.\label{sob}
\end{equation}
This together with \eqref{fen} implies
\begin{eqnarray}
\irnt|u|^{\frac{4}{N}+2}dxdt	&\leq&\int_{0}^{T}\left(\irn |u|^2dx\right)^{\frac{2}{N}}\left(\irn|u|^{\frac{2N}{N-2}}dx\right)^{\frac{N-2}{N}}dt\nonumber\\
	&\leq&\sup_{0\leq t\leq T}\left(\irn |u|^2dx\right)^{\frac{2}{N}}\int_{0}^{T}\left(\irn|u|^{\frac{2N}{N-2}}dx\right)^{\frac{N-2}{N}}dt\nonumber\\
	&\leq&c(N)\sup_{0\leq t\leq T}\left(\irn |u|^2dx\right)^{\frac{2}{N}}\irnt|\nabla u|^2dxdt\nonumber\\
	&\leq& c\left[\irn\left(|\uo|^2+|\bo|^2\right)dx\right]^{1+\frac{2}{N}}.\nonumber
\end{eqnarray}
By the same token, we also have
\begin{equation}
	\irnt|b|^{\frac{4}{N}+2}dxdt\leq c\left[\irn\left(|\uo|^2+|\bo|^2\right)dx\right]^{1+\frac{2}{N}}.\nonumber
\end{equation}
Combing the preceding two estimates yields \eqref{wi1}.

\begin{proof}[Proof of Theorem \ref{thm}]
	 Set
	\begin{equation}\label{ard}
		A_r=\|w\|_{r,Q_T}, \ \ r>1. 	
	\end{equation}
	Consider the functions
	\begin{equation}\label{aw1}
		\vp=\frac{|u|^2}{A_r},\ \ \psi=\frac{|b|^2}{A_r}.
	\end{equation}
	We proceed to derive the differential inequalities satisfied by these two functions. To this end, 
	we take the dot product of \eqref{nsf1} with $u$ to obtain
	\begin{equation*}
		\pt u\cdot u+(u\cdot\nabla) u\cdot u-(b\cdot\nabla) b\cdot u+\nabla p\cdot u=\nu\Delta u\cdot u.
	\end{equation*}
	Incorporating \eqref{am1}-\eqref{am4} into the equation, we arrive at
	\begin{equation}
		\pt|u|^2+u\cdot\nabla|u|^2+2\nu|\nabla u|^2-\nu\Delta|u|^2=2\nabla\cdot(b\otimes bu)-2(b\cdot\nabla) u\cdot b-2\nabla\cdot(pu).\nonumber
	\end{equation}
	Divide through the equation by $A_r$ to obtain
	\begin{equation}\label{use}
		\pt\vp+u\cdot\nabla\vp+2A_r^{-1}\nu|\nabla u|^2-\nu\Delta\vp=2A_r^{-1}\nabla\cdot(b\otimes bu)-2A_r^{-1}(b\cdot\nabla) u\cdot b-2A_r^{-1}\nabla\cdot(pu).
	\end{equation}
	Similarly, take the dot product of \eqref{nsf2} with b to get
	\begin{equation}
		\pt \psi+u\cdot\nabla\psi+2A_r^{-1}\eta|\nabla b|^2-\eta\Delta\psi=2A_r^{-1}\nabla\cdot(b\otimes bu)-2A_r^{-1}(b\cdot\nabla) b\cdot u.\nonumber
	\end{equation}
	
	Next, we show that we can apply  a De Giorgi iteration scheme to \eqref{use}. 
This will be done in a way similar to \cite{X1}. 
	Select 
	\begin{equation}\label{kcon1}
		k\geq 2\|\vp(\cdot,0)+\psi(\cdot,0)\|_{\infty,\rn}
	\end{equation}
	as below.	Define
	\begin{eqnarray*}
		k_n&=&k-\frac{k}{2^{n+1}} \ \ \mbox{for $n=0,1,\cdots$.} 
	\end{eqnarray*}
	Obviously, 
	\begin{equation}
		\mbox{$\{k_n\}$ is increasing and $\frac{k}{2}\leq k_n\leq k$ for each $n$.}\nonumber
	\end{equation} 
	Fix 
	\begin{equation}
		\beta>1.\nonumber
	\end{equation}
	We can easily verify that
	\begin{equation}
		\left(\frac{1}{k_n^\beta}-\frac{1}{\vp^\beta}\right)^+\nonumber
	\end{equation}
	is a legitimate test function for \eqref{use}. Upon using it, we derive
	\begin{eqnarray}
		\lefteqn{\frac{d}{dt}\irn\int_{k_n}^{\vp}	\left(\frac{1}{k_n^\beta}-\frac{1}{s^\beta}\right)^+dsdx +\nu\beta\int_{\{\vp(\cdot,t)\geq k_{n}\}}\frac{|\nabla\vp|^2}{\vp^{\beta+1}}dx}\nonumber\\
		&&+2A_r^{-1}\nu\irn|\nabla u|^2	\left(\frac{1}{k_n^\beta}-\frac{1}{\vp^\beta}\right)^+dx\nonumber\\
		&=&-\irn u\cdot\nabla\vp\left(\frac{1}{k_n^\beta}-\frac{1}{\vp^\beta}\right)^+dx-2A_r^{-1}\beta\int_{\{\vp(\cdot,t)\geq k_{n}\}}\frac{ b\otimes bu\cdot\nabla\vp}{\vp^{\beta+1}}dx\nonumber\\
		&&-2A_r^{-1}\irn(b\cdot\nabla)u\cdot b\left(\frac{1}{k_n^\beta}-\frac{1}{\vp^\beta}\right)^+dx+2A_r^{-1}\beta\int_{\{\vp(\cdot,t)\geq k_{n}\}}\frac{pu\cdot\nabla\vp}{\vp^{\beta+1}}dx.\label{pub1}
	\end{eqnarray}
	We proceed to estimate each term on the right-hand of the above inequality. We begin with the first term.
	Invoking \eqref{am6} yields
	\begin{eqnarray*}
		-\irn u\cdot\nabla\vp\left(\frac{1}{k_n^\beta}-\frac{1}{\vp^\beta}\right)^+dx=-\irn u\cdot\nabla\int_{k_n}^{\vp}\left(\frac{1}{k_n^\beta}-\frac{1}{s^\beta}\right)^+dsdx=0.\nonumber
	\end{eqnarray*}
	The second and the last term can be estimated via Young's inequality \eqref{you} as follows:
	\begin{eqnarray}
		-2A_r^{-1}\beta\int_{\{\vp(\cdot,t)\geq k_{n}\}}\frac{ b\otimes bu\cdot\nabla\vp}{\vp^{\beta+1}}dx	&\leq& \frac{\nu\beta}{4}\int_{\{\vp(\cdot,t)\geq k_{n}\}}\frac{|\nabla\vp|^2}{\vp^{\beta+1}}dx+\frac{4\beta}{\nu A_r^{2}}\int_{\{\vp(\cdot,t)\geq k_{n}\}}\frac{|b|^4|u|^2}{\vp^{\beta+1}}dx\nonumber\\
		&\leq& \frac{\nu\beta}{4}\int_{\{\vp(\cdot,t)\geq k_{n}\}}\frac{|\nabla\vp|^2}{\vp^{\beta+1}}dx+\frac{4\beta}{\nu A_r k_n^\beta}\int_{\{\vp(\cdot,t)\geq k_{n}\}}w^2dx,\nonumber\\
		2A_r^{-1}\beta\int_{\{\vp(\cdot,t)\geq k_{n}\}}\frac{pu\cdot\nabla\vp}{\vp^{\beta+1}}dx
		&\leq&\frac{\nu\beta}{4}\int_{\{\vp(\cdot,t)\geq k_{n}\}}\frac{|\nabla\vp|^2}{\vp^{\beta+1}}dx+\frac{4\beta}{\nu A_r k_n^\beta}\int_{\{\vp(\cdot,t)\geq k_{n}\}}p^2dx.\nonumber
	\end{eqnarray}
	The third term can be treated in a similar manner. That is, we have
	\begin{eqnarray}
		\lefteqn{-2A_r^{-1}\irn(b\cdot\nabla)u\cdot b\left(\frac{1}{k_n^\beta}-\frac{1}{\vp^\beta}\right)^+dx}	\nonumber\\
		&\leq&2A_r^{-1}\nu\irn|\nabla u|^2	\left(\frac{1}{k_n^\beta}-\frac{1}{\vp^\beta}\right)^+dx+\frac{1}{2\nu A_r}\irn|b|^4\left(\frac{1}{k_n^\beta}-\frac{1}{\vp^\beta}\right)^+dx	\nonumber\\
		&\leq&2A_r^{-1}\nu\irn|\nabla u|^2	\left(\frac{1}{k_n^\beta}-\frac{1}{\vp^\beta}\right)^+dx+\frac{1}{2\nu A_r k_n^\beta}\int_{\{\vp(\cdot,t)\geq k_{n}\}}w^2dx.\nonumber
	\end{eqnarray}
	Collecting the preceding estimates in \eqref{pub1}, we arrive at
	\begin{eqnarray}
		\lefteqn{\frac{d}{dt}\irn\int_{k_n}^{\vp}	\left(\frac{1}{k_n^\beta}-\frac{1}{s^\beta}\right)^+dsdx +\frac{\nu\beta}{2}\int_{\{\vp(\cdot,t)\geq k_{n}\}}\frac{|\nabla\vp|^2}{\vp^{\beta+1}}dx}\nonumber\\
		&\leq&	\frac{9\beta}{2\nu A_r k_n^\beta}\int_{\{\vp(\cdot,t)\geq k_{n}\}}w^2dx+\frac{4\beta}{\nu A_r k_n^\beta}\int_{\{\vp(\cdot,t)\geq k_{n}\}}p^2dx.\nonumber
	\end{eqnarray}
	Integrate with respect to $t$ to obtain
	\begin{eqnarray}
		\sup_{0\leq t\leq T}\irn\int_{k_n}^{\vp}	\left(\frac{1}{k_n^\beta}-\frac{1}{s^\beta}\right)^+dsdx+\int_{\{\vp\geq k_{n}\}}\frac{|\nabla\vp|^2}{\vp^{\beta+1}}dxdt\leq \frac{cI}{A_r k_n^\beta},\label{hope10}
	\end{eqnarray}
	where $c=c(\nu,\beta)$ and
	\begin{eqnarray}
		I_n&=&\int_{\{\vp\geq k_{n}\}}w^2dxdt+\int_{\{\vp\geq k_{n}\}}p^2dxdt.\nonumber
	\end{eqnarray}
	We easily verify
	\begin{eqnarray*}
		\int_{\{\vp\geq k_{n}\}}\frac{|\nabla\vp|^2}{\vp^{\beta+1}}dxdt	&=&\frac{4}{(\beta-1)^2}\int_{Q_T}\left|\nabla\left(\frac{1}{ k_{n}^{\frac{\beta-1}{2}}}-\frac{1}{\vp ^{\frac{\beta-1}{2}}}\right)^+\right|^2dxdt.
	\end{eqnarray*}
	Moreover,
	\begin{equation}
		\int_{k_{n}}^{\vp }\left(\frac{1}{k_{n}^\beta}-\frac{1}{\mu^\beta}\right)^+d\mu\geq \frac{2\beta}{(1-\beta)^2} \left[\left(\frac{1}{k_{n}^{\frac{\beta-1}{2}}}-\frac{1}{\vp ^{\frac{\beta-1}{2}}}\right)^+\right]^2.\nonumber
	\end{equation}
	Incorporating them into \eqref{hope10} yields
	\begin{equation}
		\sup_{0\leq t\leq T}\irn\left[\left(\frac{1}{k_{n}^{\frac{\beta-1}{2}}}-\frac{1}{\vp ^{\frac{\beta-1}{2}}}\right)^+\right]^2dx+\int_{Q_T}\left|\nabla\left(\frac{1}{ k_{n}^{\frac{\beta-1}{2}}}-\frac{1}{\vp ^{\frac{\beta-1}{2}}}\right)^+\right|^2dxdt\leq \frac{cI_n}{A_r k_n^\beta}.\nonumber
	\end{equation}
	%
	We calculate, with the aid of \eqref{sob}, that
	\begin{eqnarray}
		\lefteqn{	\irnt\left[\left(\frac{1}{ k_{n}^{\frac{\beta-1}{2}}}-\frac{1}{\vp ^{\frac{\beta-1}{2}}}\right)^+\right]^{\frac{4}{N}+2}dxdt}\nonumber\\
		&\leq&\int_{0}^{T}\left(\irn\left[\left(\frac{1}{ k_{n}^{\frac{\beta-1}{2}}}-\frac{1}{\vp ^{\frac{\beta-1}{2}}}\right)^+\right]^{2}dx \right)^{\frac{2}{N}}\left(\irn\left[\left(\frac{1}{ k_{n}^{\frac{\beta-1}{2}}}-\frac{1}{\vp ^{\frac{\beta-1}{2}}}\right)^+\right]^{\frac{2N}{N-2}}dx\right)^{\frac{N-2}{N}}dt\nonumber\\
		&\leq& c\left(\sup_{0\leq t\leq T}\irn\left[\left(\frac{1}{ k_{n}^{\frac{\beta-1}{2}}}-\frac{1}{\vp ^{\frac{\beta-1}{2}}}\right)^+\right]^{2}dx \right)^{\frac{2}{N}}\irnt\left|\nabla\left(\frac{1}{ k_{n}^{\frac{\beta-1}{2}}}-\frac{1}{\vp ^{\frac{\beta-1}{2}}}\right)^+\right|^2dxdt\nonumber\\
		&\leq& c\left(\frac{I_n}{A_r k_n^\beta}\right)^{\frac{N+2}{N}}.\label{rub3}
	\end{eqnarray}
	It is easy to verify that
	\begin{eqnarray*}
		\irnt\left[\left(\frac{1}{ k_{n}^{\frac{\beta-1}{2}}}-\frac{1}{\vp ^{\frac{\beta-1}{2}}}\right)^+\right]^{\frac{4}{N}+2}dxdt	&\geq&
		\int_{\{\vp\geq k_{n+1}\}}\left[\left(\frac{1}{ k_{n}^{\frac{\beta-1}{2}}}-\frac{1}{\vp ^{\frac{\beta-1}{2}}}\right)^+\right]^{\frac{4}{N}+2}dxdt\nonumber\\
		&\geq&\left(\frac{1}{ k_{n}^{\frac{\beta-1}{2}}}-\frac{1}{k_{n+1}^{\frac{\beta-1}{2}}}\right)^{\frac{4}{N}+2}|\{\vp\geq k_{n+1}\}|
		\nonumber\\
		&=&\left(\frac{\left(1-\frac{1}{2^{n+2}}\right)^{\frac{\beta-1}{2}}-\left(1-\frac{1}{2^{n+1}}\right)^{\frac{\beta-1}{2}}}{ k^{\frac{\beta-1}{2}}\left(1-\frac{1}{2^{n+1}}\right)^{\frac{\beta-1}{2}}\left(1-\frac{1}{2^{n+2}}\right)^{\frac{\beta-1}{2}}}\right)^{\frac{4}{N}+2}|\{\vp\geq k_{n+1}\}|
		\nonumber\\
		&	\geq&\frac{c|\{\vp\geq k_{n+1}\}|}{2^{(\frac{4}{N}+2)n}k^{\frac{(\beta-1)(N+2)}{N}}}.
	\end{eqnarray*}
	Combining this with \eqref{rub3} yields
	\begin{equation}
		|\{\vp\geq k_{n+1}\}|^{\frac{N}{N+2}}\leq \frac{c4^nI_n}{A_r k} .\nonumber
	\end{equation}
	By the same token, we can establish
	\begin{equation}
		|\{\psi\geq k_{n+1}\}|^{\frac{N}{N+2}}\leq \frac{c4^nJ_n}{A_r k} ,\nonumber
	\end{equation}
	where
	\begin{equation}
		J_n=\int_{\{\psi\geq k_{n}\}}w^2dxdt.\nonumber
	\end{equation}
	Set 
	\begin{equation}
		y_n=|\{\vp\geq k_{n}\}|+|\{\psi\geq k_{n}\}|.\nonumber
	\end{equation}
	Then we easily see that
	\begin{eqnarray}\label{yn2}
		y_{n+1} &=&\left(|\{\vp\geq k_{n+1}\}|+|\{\psi\geq k_{n+1}\}|\right)^{\frac{N}{N+2}+\frac{2}{N+2}}\nonumber\\
		&\leq&\frac{c4^n(I_n+J_n)}{A_r k} y_n^{\frac{2}{N+2}}.
	\end{eqnarray}
	Now we turn our attention to $p$. Take the divergence of both sides of \eqref{nsf1} to obtain
	\begin{equation*}
		-\Delta p=\nabla \cdot\left((u\cdot\nabla) u-(b\cdot\nabla) b\right).
	\end{equation*}
	In view of the classical representation theorem (\cite{GT}, p. 17), we can write $p$ as 
	\begin{equation*}
		p(x,t)=\irn \Gamma(y-x)\left[(u\cdot\nabla) u-(b\cdot\nabla) b\right]dy.
	\end{equation*} 
	We observe from \eqref{nsf4} that
	\begin{equation*}
		\irn \Gamma(y-x)\nabla\cdot((u\cdot\nabla) u-(b\cdot\nabla) b)dy=\irn \Gamma_{y_iy_j}(y-x)\left(u_iu_j-b_ib_j\right) dy.
	\end{equation*}
	It is a well known fact that $\partial^2_{y_iy_j}\Gamma(y)$ is a Calder\'{o}n-Zygmund kernel. A result of \cite{CFL} asserts that for each $ s \in (1,\infty)$ there is a positive number $c_s$ determined by $N$ and $s $ such that 
	\begin{equation*}
		\left\|\irn \Gamma_{y_iy_j}(y-x)\left(u_iu_j-b_ib_j\right) dy\right\|_{s,\rn}\leq c_s\|u_iu_j-b_ib_j\|_{s,\rn}\leq c\|w\|_{s,\rn}.
	\end{equation*}
	As a result, we have
	\begin{equation}\label{use1}
		\|p\|_{s,\rn}\leq c_s\|w\|_{s,\rn}\ \ \mbox{for each $s>1$}.
	\end{equation}
	We pick 
	\begin{equation}\label{qcon}
		q>N+2.
	\end{equation}
	Subsequently,  \eqref{use1} asserts
	\begin{equation*}
		\int_{\{\vp\geq k_{n}\}}p^2dxdt\leq \|p\|_{q,Q_T}^2|\{\vp\geq k_{n}\}|^{1-\frac{2}{q}}\leq c\|w\|_{q,Q_T}^2|\{\vp\geq k_{n}\}|^{1-\frac{2}{q}}.
	\end{equation*}
	Similarly,
	\begin{eqnarray*}
		\int_{\{\vp\geq k_{n}\}}w^2dxdt&\leq &\|w\|_{q,Q_T}^2\{\vp\geq k_{n}\}|^{1-\frac{2}{q}},\\
		\int_{\{\psi\geq k_{n}\}}w^2dxdt&\leq &\|w\|_{q,Q_T}^2\{\psi\geq k_{n}\}|^{1-\frac{2}{q}}.
	\end{eqnarray*}
	Combining them yields
	\begin{equation}
		I_n+J_n\leq c\|w\|_{q,Q_T}^2y_n^{1-\frac{2}{q}}.\nonumber
	\end{equation}
	Incorporate this into \eqref{yn2} to get
	\begin{equation}\label{hap1}
		y_{n+1}\leq \frac{c4^n\|w\|_{q,Q_T}^2}{A_r k} y_n^{1+\alpha}. 
	\end{equation}
	where
	\begin{equation}\label{adef}
		\alpha=\frac{2}{N+2}-\frac{2}{q}=\frac{2(q-N-2)}{(N+2)q} >0\ \ \mbox{due to  \eqref{qcon}}.
	\end{equation}
	At this point, our constant $c$ in \eqref{hap1} depends on $q, N, \nu, \eta$. In the classical  De Giorgi iteration scheme, one just applies Lemma \ref{ynb} to \eqref{hap1}, which will lead to an estimate like \eqref{int2}. The result  is not useful here.  What makes our method different is the introduction of the following two parameters $\ell$ and $j$. The condition on $\ell$ is
	\begin{equation}
		\ell>r.\nonumber
	\end{equation}
	We further require
	\begin{equation}\label{qcl}
		q>\lz.
	\end{equation}
	Choose $k$ so large that
	\begin{equation}\label{kcon6}
		\max\left\{L_1\left\|\vp+\psi\right\|_{\ell,Q_T}^{\frac{\ell}{\ell-r}}, L_2\|w\|_{r,Q_T}^{-1}\|w\|_{ q,Q_T}^{\frac{q}{q-\lz}}\right\}	\leq k,
	\end{equation}
	where $L_1$ and $L_2$ are two positive numbers to be determined.  
	Consequently, for each 
	\begin{equation*}
		j>0
	\end{equation*}
	there hold
	\begin{equation*}
		L_1^{j\alpha}\left\|\vp+\psi\right\|_{\ell,Q_T}^{\frac{j\alpha\ell}{\ell-r}}\leq k^{j\alpha}	\ \ \mbox{and}\ \ L_2^{j\alpha+1}\|w\|_{r,Q_T}^{-(j\alpha+1)}\|w\|_{ q,Q_T}^{\frac{(j\alpha+1) q}{q-\lz}}\leq k^{j\alpha+1}.	
	\end{equation*}
	Use them in \eqref{hap1} and keep \eqref{ard} in mind to deduce
	\begin{eqnarray*}
		y_{n+1}&\leq&	\frac{c4^n\|w\|_{r,Q_T}^{j\alpha}\|w\|_{ q,Q_T}^{b} k^{2j\alpha}}{L_1^{j\alpha}L_2^{j\alpha+1}\left\|\vp+\psi\right\|_{\ell,Q_T}^{\frac{j\alpha\ell}{\ell-r}}} y_n^{1+\alpha}
	\end{eqnarray*}
	where
	\begin{equation}\label{bdef}
		b=2-\frac{ (j\alpha+1)q}{q-\lz}=\frac{q-2\lz}{q-\lz}-\frac{ j\alpha q}{q-\lz}.
	\end{equation}
	The introduction of $j$ here is also very crucial. As we shall see, by choosing $j$ suitably large, we can make certain exponents in our nonlinear terms negative. This will enable us to balance out large positive exponents. 
	
	To apply Lemma \ref{ynb}, we first recall \eqref{ard} and \eqref{aw1} to deduce
	\begin{eqnarray*}
		y_0&=&\left|\left\{\vp>\frac{k}{2}\right\}\right|+\left|\left\{\psi>\frac{k}{2}\right\}\right|\nonumber\\
		&\leq& \irnt\left(\frac{2\vp}{k}\right)^{r}dxdt+\irnt\left(\frac{2\psi}{k}\right)^{r}dxdt\nonumber\\
		&\leq&\frac{2^{r+1}}{k^r}\irnt\left(\vp+\psi\right)^rdxdt\leq \frac{2^{r+1}}{k^r}.
	\end{eqnarray*}
	Assume that
	\begin{equation}\label{rcon}
		r>2j.
	\end{equation}
	Subsequently, we can pick $k$ so large that
	\begin{eqnarray}
		\frac{2^{r+1}}{k^{r-2j}}\leq \frac{L_1^{j}L_2^{\frac{\alpha j+1}{\alpha}}\left\|\vp+\psi\right\|_{\ell,Q_T}^{\frac{j\ell}{\ell-r}}}{ c^{\frac{1}{\alpha}}4^{\frac{1}{\alpha^2}}\|w\|_{ r,Q_T}^{j}\|w\|_{ q,Q_T}^{\frac{b}{\alpha}}}.\label{jm7}
	\end{eqnarray}
	Lemma \ref{ynb} asserts 
	$$\lim_{n\ra \infty}y_n=|\{\vp\geq k\}|+|\{\psi\geq k\}|=0.$$ 
	That is,
	\begin{equation}\label{es1}
		\sup_{Q_T}(	\vp+\psi)\leq 2 k.
	\end{equation}
	In view of \eqref{kcon1},  \eqref{kcon6}, and \eqref{jm7}, it is enough for us to take
	\begin{eqnarray}
		k&=&2\|\vp(\cdot,0)+\psi(\cdot,0)\|_{\infty,\rn}+L_1 \left\|\vp+\psi\right\|_{\ell,Q_T}^{\frac{\ell}{\ell-r}}+L_2\|w\|_{r,Q_T}^{-1}\|w\|_{ q,Q_T}^{\frac{q}{q-\lz}}\nonumber\\ &&+ c_14^{-\frac{j}{r-2j}}
		L_1^{-\frac{j}{r-2j}}L_2^{-\frac{\alpha j+1}{\alpha(r-2j)}}\left\|\vp+\psi\right\|_{\ell,Q_T}^{-\beta_1}\|w\|_{ r,Q_T}^{\frac{j}{r-2j}}\|w\|_{ q,Q_T}^{\frac{b}{\alpha\left(r-2j\right)}},\nonumber
	\end{eqnarray}
	where
	\begin{eqnarray}
		c_1&=&2^{\frac{r+1}{r-2j}}c^{\frac{1}{\alpha\left(r-2j\right)}}4^{\frac{1}{\alpha^2\left(r-2j\right)}+\frac{j}{r-2j}},\label{cr}\\
		\beta_1&=&\frac{j\ell}{\left(r-2j\right)(\ell-r)}.\label{b1}
	\end{eqnarray}
	Remember that the constant $c$ on the right-hand side of \eqref{cr} is independent of $r, j$, and $\ell$.
	Now \eqref{es1} becomes
	\begin{eqnarray}
		\left\|\vp+\psi\right\|_{\infty,Q_T}&\leq& 4\|\vp(\cdot,0)+\psi(\cdot,0)\|_{\infty,\rn}+2L_1 \left\|\vp+\psi\right\|_{\ell,Q_T}^{\frac{\ell}{\ell-r}}+2L_2\|w\|_{r,Q_T}^{-1}\|w\|_{ q,Q_T}^{\frac{q}{q-\lz}}\nonumber\\ &&+ c_14^{-\frac{j}{r-2j}}
		L_1^{-\frac{j}{r-2j}}L_2^{-\frac{\alpha j+1}{\alpha(r-2j)}}\left\|\vp+\psi\right\|_{\ell,Q_T}^{-\beta_1}\|w\|_{ r,Q_T}^{\frac{j}{r-2j}}\|w\|_{ q,Q_T}^{\frac{b}{\alpha\left(r-2j\right)}}.\label{jm8}
	\end{eqnarray}
	In view of \eqref{ard} and \eqref{aw1}, we have 
	\begin{eqnarray}
		\left\|\vp+\psi\right\|	_{\ell,Q_T}^{\frac{\ell}{\ell-r}}&=&\left\|\frac{w}{\wnr}\right\|	_{\ell,Q_T}^{\frac{\ell}{\ell-r}}\nonumber\\
		&\leq&\frac{1}{\wnr^{\frac{\ell}{\ell-r}}}\left(\left\|w\right\|	_{\infty,Q_T}^{\frac{\ell-r}{\ell}}\left\|w\right\|	_{r,Q_T}^{\frac{r}{\ell}}\right)^{\frac{\ell}{\ell-r}}
		=\left\|\vp+\psi\right\|	_{\infty,Q_T} .\nonumber
	\end{eqnarray}
	Plug this into \eqref{jm8} and take $ L_1=\frac{1}{4}$ in the resulting inequality
	to yield
	\begin{eqnarray*}
		\left\|\vp+\psi\right\|_{\infty,Q_T}&\leq& 8\|\vp(\cdot,0)+\psi(\cdot,0)\|_{\infty,\rn}+4L_2\|w\|_{r,Q_T}^{-1}\|w\|_{ q,Q_T}^{\frac{q}{q-\lz}}\nonumber\\ &&+ 2c_1
		L_2^{-\frac{\alpha j+1}{\alpha(r-2j)}}\left\|\vp+\psi\right\|_{\ell,Q_T}^{-\beta_1}\|w\|_{ r,Q_T}^{\frac{j}{r-2j}}\|w\|_{ q,Q_T}^{\frac{b}{\alpha\left(r-2j\right)}}.
	\end{eqnarray*}
	Multiply through the inequality by $A_r$ and keep \eqref{ard} and \eqref{aw1} in mind to deduce
	\begin{eqnarray}
		\|w\|_{\infty,Q_T}&\leq&8 \|w(\cdot,0)\|_{\infty,\rn}+4L_2\|w\|_{ q,Q_T}^{\frac{q}{q-\lz}}\nonumber\\ &&+	2c_1L_2^{-\frac{\alpha j+1}{\alpha(r-2j)}}\left\|w\right\|_{\ell,Q_T}^{-\beta_1}\|w\|_{ r,Q_T}^{1+\frac{j}{r-2j}+\beta_1}\|w\|_{ q,Q_T}^{\frac{b}{\alpha\left(r-2j\right)}} .\label{est1}
	\end{eqnarray}
	By \eqref{qcl}, we have the interpolation inequality
	\begin{equation}\label{intq}
		\|w\|_{q,Q_T}\leq \|w\|_{\infty,Q_T}^{\frac{q-\lz}{q}}\|w\|_{\lz,Q_T}^{\frac{\lz}{q}}.
	\end{equation}
	Choose $L_2$  so that
	\begin{equation*}
		4L_2\|w\|_{\lz,Q_T}^{\frac{\lz}{q-\lz}}= \frac{1}{2}.
	\end{equation*}
	Incorporate this into \eqref{est1} to deduce
	\begin{eqnarray}
		\|w\|_{\infty,Q_T}&\leq& 16\|w(\cdot,0)\|_{\infty,\rn}+ c_2\|w\|_{\lz,Q_T}^{s_1}\|w\|_{\ell,Q_T}^{-\beta_1}\|w\|_{r,Q_T}^{1+\frac{j}{r-2j}+\beta_1}\|w\|_{q,Q_T}^{\frac{b}{\alpha\left(r-2j\right)}},\label{est2}
	\end{eqnarray}
	where
	\begin{equation}\label{c2}
		c_2=4c_18^{\frac{\alpha j+1}{\alpha(r-2j)}}=c^{\frac{1}{\alpha\left(r-2j\right)}}4^{\frac{1}{\alpha^2\left(r-2j\right)}+\frac{r-j}{r-2j}+\frac{2(\alpha j+1)}{\alpha(r-2j)}+\frac{r+1}{2(r-2j)}},\ \ s_1=\frac{\lz(\alpha j+1) }{\alpha(q-\lz)\left(r-2j\right)}.
	\end{equation}
	In summary, we have shown:
	\begin{clm}
	For each $q$ satisfying \eqref{qcon} there is a constant $c=c(N, q,\nu,\eta )$ such that \eqref{est2} holds for each
		\begin{equation}
			j>0,\ \ r>2j, \ \ \mbox{and} \ \ \ell>r.\nonumber
		\end{equation}
			\end{clm}
	We wish to transform the last three different norms in \eqref{est2} into a single one. Our first attempt yields the following result.
	\begin{clm} Let $q$ be given as in \eqref{qcon}. We further assume
		\begin{equation}\label{qc}
			q> 2\lz.
		\end{equation}
			Then for each
		\begin{equation}
		r>\mq\equiv\frac{q-2\lz}{\alpha q}
			,\label{mqd}
		\end{equation}
	  there exist two positive numbers $c_3, s_0$ such that	
		\begin{eqnarray}
			\|w\|_{\infty,Q_T}&\leq& 16\|w(\cdot,0)\|_{\infty,\rn}+ c_3\|w\|_{\lz,Q_T}^{s_0}\|w\|_{r,Q_T}^{\frac{r}{r-\mq}}.\label{ha7}
		\end{eqnarray}
		Moreover, the dependence of $c_3$ and $s_0$ on $r$ is such that both $\ln c_3$ and $s_0$ are smooth functions of $r$ over $(\mq, \infty)$.
	\end{clm}
	\begin{proof}
		Insert \eqref{b1} into  \eqref{est2} and take $\ell\ra \infty$ in the resulting inequality to obtain 
		\begin{eqnarray}
			\|w\|_{\infty,Q_T}&\leq& 16\|w(\cdot,0)\|_{\infty,\rn}+c_2\|w\|_{\lz,Q_T}^{s_1}\|w\|_{\infty,Q_T}^{-\frac{j}{r-2j}}\|w\|_{r,Q_T}^{\frac{r}{r-2j}}\|w\|_{q,Q_T}^{\frac{b}{\alpha\left(r-2j\right)}} .\label{est13}
		\end{eqnarray}
		Here the constants $c_2$ and $s_1$ remain the same as those in \eqref{est2} because they are independent of $\ell$.
		Under \eqref{qc},  inequality \eqref{intq} remains valid. 
		Raise both side of the inequality to the power of $\frac{jq}{(r-2j)(q-\lz)}$ to deduce
		\begin{equation}
			\|w\|_{q,Q_T}^{\frac{jq}{(r-2j)(q-\lz)}}\leq \|w\|_{\infty,Q_T}^{\frac{j}{r-2j}}\|w\|_{\lz,Q_T}^{\frac{j\lz}{(r-2j)(q-\lz)}}.\nonumber
		\end{equation}	
		Combining this with \eqref{est13} yields
		\begin{eqnarray}
			\|w\|_{\infty,Q_T}&\leq& 16\|w(\cdot,0)\|_{\infty,\rn}+ c_2\|w\|_{\lz,Q_T}^{s_2}\|w\|_{r,Q_T}^{\frac{r}{r-2j}}\|w\|_{q,Q_T}^{\frac{b}{\alpha\left(r-2j\right)}-\frac{jq}{(r-2j)(q-\lz)}}.\label{ha6}
		\end{eqnarray}
		where 
		\begin{equation}\label{s2}
			s_2=s_1+\frac{j\lz}{(r-2j)(q-\lz)}.
		\end{equation}
		We pick $j$ so that
		\begin{equation}\label{je}
			\frac{b}{\alpha\left(r-2j\right)}-\frac{jq}{(r-2j)(q-\lz)}=0.
		\end{equation}
		Plug \eqref{bdef} into this to derive
		\begin{equation}
			\frac{q-2\lz-2\alpha jq}{\alpha(r-2j)(q-\lz)}=0,\nonumber
		\end{equation}
		which means that we can achieve \eqref{je} by taking
		\begin{equation}
			j=\frac{q-2\lz}{2\alpha q}=\frac{\mq}{2}.\nonumber
		\end{equation}
	We arrive at \eqref{ha7} by inserting the preceding choice of $j$ into \eqref{ha6}. In view of  \eqref{c2} and \eqref{s2}, we have
	\begin{eqnarray}
		s_0&=&s_2|_{j=\frac{\mq}{2}}=\frac{\lz(\alpha \mq+2) }{2\alpha(q-\lz)\left(r-\mq\right)}+\frac{\mq\lz}{2(r-\mq)(q-\lz)},\nonumber
		\\
		c_3&=&c_2|_{j=\frac{\mq}{2}}=c^{\frac{1}{\alpha\left(r-\mq\right)}}4^{\frac{1}{\alpha^2\left(r-\mq\right)}+\frac{2r-\mq}{2(r-\mq)}+\frac{\alpha \mq+2}{\alpha(r-\mq)}+\frac{r+1}{2(r-\mq)}}.\label{c21}
	\end{eqnarray}
	Note that the constant $c$ on the right-hand side of \eqref{c21} is the same as the one in \eqref{cr}, and hence it depends on $N,q, \nu, \eta$ only. The proof is complete.
	\end{proof}

	We easily see from \eqref{mqd} and \eqref{adef} that
	\begin{equation}
		\lim_{q\ra \infty}\mq=\frac{N+2}{2}.\nonumber
	\end{equation}
	Consequently,	if
	\begin{equation}\label{lz1}
		\lz>\frac{N+2}{2}
	\end{equation}
	we may pick $q$ so that
	\begin{equation}
		\lz>\mq.\nonumber
	\end{equation}		
	This implies that we can take $r=\lz$ in \eqref{ha7}. That is to say, $w\in L^\infty(Q_T)$ whenever  $ w\in L^{\lz}(Q_T)$ with $\lz$ satisfying \eqref{lz1}.
	
	From here on we will assume  
	\begin{equation}\label{lz}
		\lz< \frac{N+2}{2}.
	\end{equation}
	As a result, we have \eqref{qc}. That is to say, \eqref{ha7} is true under the assumptions \eqref{qcon}, \eqref{lz}, and \eqref{mqd}.
 Moreover,  plug \eqref{adef} into the definition of $\mq$ in \eqref{mqd} to derive 
	\begin{equation} \label{ml}
		\mq=\frac{(N+2)(q-2\lz)}{2(q-N-2)}>\frac{N+2}{2}>\lz.
	\end{equation}
	 We can also represent $b$ in \eqref{bdef} as
	\begin{equation}
		b=\frac{\alpha q(\mq-j)}{q-\lz}.\nonumber
	\end{equation}
	Plugging this and \eqref{b1} into \eqref{est2}, we arrive at
		\begin{eqnarray}
		\|w\|_{\infty,Q_T}&\leq& 16\|w(\cdot,0)\|_{\infty,\rn}\nonumber\\
		&&+ c_2\|w\|_{\lz,Q_T}^{s_1}\|w\|_{\ell,Q_T}^{-\frac{j\ell}{(r-2j)(\ell-r)}}\|w\|_{r,Q_T}^{1+\frac{j}{r-2j}+\frac{j\ell}{(r-2j)(\ell-r)}}\|w\|_{q,Q_T}^{\frac{q(\mq-j)}{(q-\lz)\left(r-2j\right)}}.\label{est3}
	\end{eqnarray}
	
	We wish to improve \eqref{ha7} by making better choice of the parameters in \eqref{est3}. 
	What motivates us to do so is the observation that the sum of the last three exponents in \eqref{est3} is given by
	\begin{equation}
		1+\frac{j}{r-2j}+\frac{q(\mq-j)}{(q-\lz)\left(r-2j\right)}=1+\frac{-\lz j+\mq q}{(r-2j)(q-\lz)}.\nonumber
	\end{equation} 
 Obviously, the total sum can be made smaller than 1 if we choose $j$ suitably large.	This suggests that one should try to transform the last three norms in \eqref{est3} into a single one  in such a way that the total sum is largely preserved.
To this end, we consider the function
	\begin{equation}
		f(r)=\ln\left(\iqt w^rdxdt\right)\ \ \mbox{for $r\geq \lz$}.\nonumber
	\end{equation}
	Remember that $w$ is a very ``nice'' function. We easily see that
	\begin{eqnarray}
		f^\prime(r)&=&\frac{\iqt w^r\ln wdxdt}{\iqt w^rdxdt},\label{fp}\\
		f^{\prime\prime}(r)&=&\frac{\iqt w^r\ln^2 wdxdt\iqt w^rdxdt-\left(\iqt w^r\ln wdxdt\right)^2}{\left(\iqt w^rdxdt\right)^2}.\nonumber
	\end{eqnarray}
	H\"{o}lder inequality asserts that
	\begin{equation}\label{pp}
		f^{\prime\prime}(r)\geq 0.
	\end{equation}
	That is, $\fr$ is convex over $(\lz, \infty)$. We will exploit this property extensively. 	We also need the following limit:
	\begin{eqnarray}
		\lim_{\ell\ra r}\left(\frac{\|w\|_{\ell,Q_T}}{\|w\|_{r,Q_T}}\right)^{\frac{1}{\ell-r}}&=&	\lim_{\ell\ra r}e^{\frac{\ln \|w\|_{\ell,Q_T}-\ln \|w\|_{r,Q_T}}{\ell-r}}\nonumber\\
		&=&e^{\left(\frac{f(r)}{r}\right)^\prime}
		=\|w\|_{r,Q_T}^{-\frac{1}{r}}e^{\frac{\iqt w^r\ln wdxdt}{r\iqt w^rdxdt}}.\nonumber
	\end{eqnarray}
	
\begin{remark}
	We may  assume that
	\begin{equation}\label{fpp}
		\fp>0\ \ \mbox{for each $r>\mq$.}
	\end{equation}
Indeed, suppose that the above is not true, i.e., there is a $r>\mq$ such that 
\begin{equation}\label{wer}
	\fp\leq0.
\end{equation}
By \eqref{fp}, we have
\begin{eqnarray}
	\int_{\{w\geq 1\}}w^r\ln w dxdt&\leq& -	\int_{\{w< 1\}}w^r\ln w dxdt\nonumber\\
	&\leq& \max_{s\in[0,1]}	|s^{r-\frac{N+2}{N}}\ln s|\int_{\{w< 1\}}w^{\frac{N+2}{N}} dxdt\leq c.\nonumber
\end{eqnarray}
The last step here is due to \eqref{wi1}.
For each $L>1$ we have
\begin{eqnarray}
	\iqt w^rdxdt&=&	\int_{\{w\geq L\}}w^r dxdt+\int_{\{w< L\}}w^r dxdt\nonumber\\
	&\leq&\frac{1}{\ln L}\int_{\{w\geq L\}}w^r \ln wdxdt+L^{r-\frac{N+2}{N}}\int_{\{w< L\}}w^ {\frac{N+2}{N}}dxdt\leq c.\nonumber
\end{eqnarray}
This combined with \eqref{ha7} implies our theorem. That is to say, were \eqref{wer} true for some $r>\mq$, we would have nothing more to prove. In this case, the constant in \eqref{es} also depends on $\|\uo\|_{2,\rn}$ and $\|\bo\|_{2,\rn}$. By the same token, we may also assume
\begin{equation}\label{fp1}
	f(s)>0 \ \ \mbox{on $(\mq, \infty)$.}
\end{equation}
However, neither \eqref{fpp} nor \eqref{fp1} will be used in the subsequent proof.
\end{remark}

	Remember  that the constant $c_2$ in \eqref{est3} does not depend $\ell$. For simplicity, from here on we will just denote it by $c$.  With this in mind, 
	we take $\ell\ra r$ in \eqref{est3}, thereby obtaining
	\begin{eqnarray}
		\|w\|_{\infty,Q_T}&\leq& 16\|w(\cdot,0)\|_{\infty,\rn}\nonumber\\
		&&+ c\|w\|_{\lz,Q_T}^{s_1}\left[\lim_{\ell\ra r}\left(\frac{\|w\|_{\ell,Q_T}}{\|w\|_{r,Q_T}}\right)^{\frac{1}{\ell-r}}\right]^{-\frac{j\ell}{r-2j}}\|w\|_{r,Q_T}^{\frac{r-j}{r-2j}}\|w\|_{q,Q_T}^{\frac{q(\mq-j)}{(q-\lz)(r-2j)}}\nonumber\\
		&=& 16\|w(\cdot,0)\|_{\infty,\rn}\nonumber\\
		&&+ c\|w\|_{\lz,Q_T}^{s_1}\left[\|w\|_{r,Q_T}^{-1}e^{\frac{\int_{Q_T}w^r\ln wdxdt}{\int_{Q_T}w^rdxdt}}\right]^{-\frac{j}{r-2j}}\|w\|_{r,Q_T}^{\frac{r-j}{r-2j}}\|w\|_{q,Q_T}^{\frac{q(\mq-j)}{(q-\lz)(r-2j)}}\nonumber\\
		&=& 16\|w(\cdot,0)\|_{\infty,\rn}+ c\|w\|_{\lz,Q_T}^{s_1}e^{\frac{-j\fp}{\left(r-2j\right)}}\|w\|_{r,Q_T}^{\frac{r}{r-2j}}\|w\|_{q,Q_T}^{\frac{q(\mq-j)}{(q-\lz)(r-2j)}}.\nonumber\\
		&=& 16\|w(\cdot,0)\|_{\infty,\rn}+ c\|w\|_{\lz,Q_T}^{s_1}e^{\frac{-j\fp+\fr}{r-2j}}\|w\|_{q,Q_T}^{\frac{q(\mq-j)}{(q-\lz)(r-2j)}}.
		\label{est10}%
	\end{eqnarray}
To continue, we further refine our choice of parameters.	We easily see from \eqref{adef} and \eqref{mqd} that
	\begin{equation}
		q-2\mq=q-\frac{2(q-2\lz)}{\alpha q}=q-\frac{(N+2)(q-2\lz)}{(q-N-2)}.\nonumber
	\end{equation}
	By \eqref{lz}, $	q-2\mq$ is a strictly increasing function of $q$ on $(N+2,\infty)$. Clearly,  its range is $(-\infty, \infty)$. Hence, there is a unique $q_0\in 	(N+2,\infty)$ such that
	\begin{equation}
		q_0=2M_{q_0}.\nonumber
	\end{equation}
Pick
	\begin{equation}\label{qcon2}
		q>q_0,
	\end{equation}
	As a result, 
	\begin{equation}
		q>2\mq.\nonumber
	\end{equation}
	Subsequently, take
	\begin{equation}\label{jl}
		j>\frac{q}{2}.
	\end{equation}
	Introduce a new parameter
	\begin{equation}
		\ve>0.\nonumber
	\end{equation}
	Then
	select
	\begin{equation}
		r>2j+\ve.\nonumber
	\end{equation}
This choice of $r$ satisfies \eqref{rcon}. 
The introduction of $\ve$ is to ensure that $r$ stays away from $2j$ because the constant $c$ and the exponents in \eqref{est10} blow up as $r\ra 2j$.
In summary, we have
\begin{equation}\label{rq}
2\mq<	q<2j<2j+\ve<r.
\end{equation}
We can form the interpolation inequality
\begin{equation}\label{int}
	\|w\|_{2j+\ve,Q_T}\leq \|w\|_{r,Q_T}^{\frac{r(2j+\ve-q)}{(2j+\ve)(r-q)}}\|w\|_{q,Q_T}^{\frac{q(r-2j-\ve)}{(2j+\ve)(r-q)}}
\end{equation}
 On account of \eqref{rq},
\begin{equation}
	j>\mq.\nonumber
\end{equation}
We may raise both sides of \eqref{int} to the power of $\frac{(j-\mq)(2j+\ve)(r-q)}{(q-\lz)(r-2j)(r-2j-\ve)}$, thereby obtaining
\begin{equation}
	\|w\|_{2j+\ve,Q_T}^{\frac{(j-\mq)(2j+\ve)(r-q)}{(q-\lz)(r-2j)(r-2j-\ve)}}\leq \|w\|_{r,Q_T}^{\frac{r(2j+\ve-q)(j-\mq)}{(q-\lz)(r-2j)(r-2j-\ve)}}\|w\|_{q,Q_T}^{\frac{q(j-\mq)}{(q-\lz)(r-2j)}}.\nonumber
\end{equation}
Incorporating this into \eqref{est10}, we arrive at
\begin{eqnarray}
		\|w\|_{\infty,Q_T}&\leq& 16\|w(\cdot,0)\|_{\infty,\rn}\nonumber\\
		&&+ c\|w\|_{\lz,Q_T}^{s_1}e^{\frac{-j\fp+\fr}{r-2j}+\frac{(2j+\ve-q)(j-\mq)\fr}{(q-\lz)(r-2j)(r-2j-\ve)}}\|w\|_{2j+\ve,Q_T}^{-\frac{(j-\mq)(2j+\ve)(r-q)}{(q-\lz)(r-2j)(r-2j-\ve)}}.\label{ja2}
\end{eqnarray}
Fix
\begin{equation}
	\eta\in (0,1).\nonumber
\end{equation}
Without any loss of generality, we may assume
	\begin{eqnarray}
\lefteqn{-jf^\prime(s)+f(s)+\frac{(2j+\ve-q)(j-\mq)f(s)}{(q-\lz)(s-2j-\ve)}}\nonumber\\ &\geq&\left[\frac{(j-\mq)(s-q)}{(q-\lz)(s-2j-\ve)}+\frac{(s-2j)(1-\eta)}{2j+\ve-\lz}\right]f(2j+\ve) \ \ \mbox{for each $s\in (2j+\ve, r]$.}\label{pp1}
			\end{eqnarray}
	Indeed, suppose this is not true. That is, there is a $s\in (2j+\ve, r]$ such that
	\begin{eqnarray}
	\lefteqn{-jf^\prime(s)+f(s)+\frac{(2j+\ve-q)(j-\mq)f(s)}{(q-\lz)(s-2j-\ve)}}\nonumber\\ &<&\left[\frac{(j-\mq)(s-q)}{(q-\lz)(s-2j-\ve)}+\frac{(s-2j)(1-\eta)}{2j+\ve-\lz}\right]f(2j+\ve).\nonumber
\end{eqnarray}
Obviously, \eqref{ja2} holds for $r=s$. As a result, we can apply the preceding inequality to it. Upon doing so, we arrive at
\begin{eqnarray}
	\|w\|_{\infty,Q_T}&\leq& 16\|w(\cdot,0)\|_{\infty,\rn}\nonumber\\
	&&+ c\|w\|_{\lz,Q_T}^{s_1}\|w\|_{2j+\ve,Q_T}^{-\frac{(j-\mq)(2j+\ve)(s-q)}{(q-\lz)(s-2j)(s-2j-\ve)}+\left[\frac{(j-\mq)(s-q)}{(q-\lz)(s-2j-\ve)}+\frac{(s-2j)(1-\eta)}{2j+\ve-\lz}\right]\frac{2j+\ve}{s-2j}}\nonumber\\
	&=& 16\|w(\cdot,0)\|_{\infty,\rn}+ c\|w\|_{\lz,Q_T}^{s_1}\|w\|_{2j+\ve,Q_T}^{\frac{(2j+\ve)(1-\eta)}{2j+\ve-\lz}}.\label{est11}
\end{eqnarray}
As we noted earlier, $c$ and $s_1$ here remain bounded for $s\in (2j+\ve, r]$.
In view of \eqref{rq}, \eqref{qcon} and \eqref{lz}, we can form the interpolation inequality
\begin{equation}
	\|w\|_{2j+\ve,Q_T}\leq \|w\|_{\infty,Q_T}^{\frac{2j+\ve-\lz}{2j+\ve}}\|w\|_{\lz,Q_T}^{\frac{\lz}{2j+\ve}}.\nonumber
\end{equation}
Collect this  in \eqref{est11} to get
\begin{eqnarray}
	\|w\|_{\infty,Q_T}&\leq& 16\|w(\cdot,0)\|_{\infty,\rn}+c\|w\|_{\lz,Q_T}^{s_1+\frac{\lz(1-\eta)}{2j+\ve-\lz}}\|w\|_{\infty,Q_T}^{1-\eta}.\nonumber
\end{eqnarray}
Then \eqref{es} follows from a suitable application of Young's inequality \eqref{you}. That is to say, if \eqref{pp1} fails to be true, then \eqref{es} holds.

We next show that \eqref{es} remains valid under \eqref{pp1}. On account of \eqref{pp},  the convexity of $f(s)$, there holds
\begin{equation}\label{ja4}
	f^\prime(s)\geq	\frac{f(s)-f(2j+\ve)}{s-2j-\ve}\ \ \mbox{for each $s>2j+\ve$}.
\end{equation}
We may decompose
\begin{equation}
	\frac{(j-\mq)(s-q)}{(q-\lz)(s-2j-\ve)}=\frac{(j-\mq)}{(q-\lz)}+\frac{(j-\mq)(2j+\ve-q)}{(q-\lz)(s-2j-\ve)}.\nonumber
\end{equation}
Incorporate the preceding two results into \eqref{pp1} to derive
	\begin{eqnarray}
-s_jf^\prime(s)+f(s) &\geq&\left[\frac{(j-\mq)}{(q-\lz)}+\frac{(s-2j)(1-\eta)}{2j+\ve-\lz}\right]f(2j+\ve) \ \ \mbox{for each $s\in (2j+\ve, r]$.}\label{ja3}
\end{eqnarray}
where
\begin{equation}\label{sj}
	s_j= j-	\frac{(2j+\ve-q)(j-\mq)}{(q-\lz)}.
\end{equation}
We further require 
\begin{equation}\label{ja8}
	s_j>0.
\end{equation}
For the  inequality to hold, it is enough for us to take
\begin{equation}\label{eu}
	\ve<\frac{q-\lz}{j-\mq}\left[j-\frac{(2j-q)(j-\mq)}{q-\lz}\right].
\end{equation}
This is possible only when
\begin{equation}\label{ja7}
	\frac{(2j-q)(j-\mq)}{q-\lz}-	j=\frac{2j^2-(2q+2\mq-\lz)j+\mq q}{q-\lz}<0.
\end{equation}
The numerator is a quadratic function in $j$. The discriminant is given by
\begin{eqnarray}
	(2q+2\mq-\lz)^2-8\mq q&=&4q^2+4\mq^2-4\lz(q+\mq)+\lz^2\nonumber\\
	&=&(2q-\lz)^2+4\mq(\mq-\lz)>0.\nonumber
\end{eqnarray} 
The last step here is due to \eqref{ml}. According to the quadratic formula, the solution set to \eqref{ja7} is given by
\begin{equation}
	j\in (j_1, j_2),\nonumber
\end{equation}
where
\begin{eqnarray}
	j_1&=&\frac{2q+2\mq-\lz-\sqrt{(2q+2\mq-\lz)^2-8\mq q}}{4},\nonumber\\
	j_1&=&\frac{2q+2\mq-\lz+\sqrt{(2q+2\mq-\lz)^2-8\mq q}}{4}.\nonumber
\end{eqnarray}
Obviously, both $j_1$ and $j_2$ are positive. In particular,
\begin{equation}\label{j2}
	j_2\in (q, q+\mq-\lz).
\end{equation}
Under \eqref{ja8}, we can write \eqref{ja3} in the form
	\begin{eqnarray}
f^\prime(s)-\frac{f(s)}{s_j} &\leq&-\left[\frac{(j-\mq)}{(q-\lz)}+\frac{(s-2j)(1-\eta)}{2j+\ve-\lz}\right]\frac{f(2j+\ve)}{s_j} \ \ \mbox{for each $s\in (2j+\ve, r]$.}\nonumber
\end{eqnarray}
Multiply through the inequality by $e^{-\frac{s}{s_j}}$ to derive 
\begin{equation}
	\left(e^{-\frac{s}{s_j}}f(s)\right)^\prime\leq -\left[\frac{(j-\mq)}{(q-\lz)}+\frac{(s-2j)(1-\eta)}{2j+\ve-\lz}\right]\frac{f(2j+\ve)e^{-\frac{s}{s_j}}}{s_j} \ \ \mbox{for each $s\in (2j+\ve, r]$.}\nonumber
\end{equation}
Subsequently,
\begin{eqnarray}
e^{-\frac{r}{s_j}}	f(r)&=&e^{-\frac{2j+\ve}{s_j}}f(2j+\ve)+\int_{2j+\ve}^{r}	\left(e^{-\frac{s}{s_j}}f(s)\right)^\prime ds\nonumber\\
	&\leq&e^{-\frac{2j+\ve}{s_j}}f(2j+\ve)+\frac{(j-\mq)\left(e^{-\frac{r}{s_j}}-e^{-\frac{2j+\ve}{s_j}}\right)f(2j+\ve)}{(q-\lz)}\nonumber\\
	&&+\frac{(1-\eta)f(2j+\ve)}{2j+\ve-\lz}\left[(r-2j)e^{-\frac{r}{s_j}}-\ve e^{-\frac{2j+\ve}{s_j}}+s_j\left(e^{-\frac{r}{s_j}}-e^{-\frac{2j+\ve}{s_j}}\right)\right]\nonumber\\
	&=&\left[1-\frac{j-\mq}{q-\lz}-\frac{(1-\eta)(\ve+s_j)}{2j+\ve-\lz}\right]e^{-\frac{2j+\ve}{s_j}}f(2j+\ve)\nonumber\\
	&&+\left[\frac{j-\mq}{q-\lz}+\frac{(1-\eta)(r-2j+s_j)}{2j+\ve-\lz}\right]e^{-\frac{r}{s_j}}f(2j+\ve).\label{h1}
\end{eqnarray}
Recall \eqref{ja4} and \eqref{sj} to obtain
\begin{eqnarray}
	\lefteqn{-j\fp+\fr+\frac{(2j+\ve-q)(j-\mq)\fr}{(q-\lz)(r-2j-\ve)}}\nonumber\\
	&\leq&-\frac{j(f(r)-f(2j+\ve))}{r-2j-\ve}+\fr+\frac{(2j+\ve-q)(j-\mq)f(r)}{(q-\lz)(r-2j-\ve)}\nonumber\\
	&=&\frac{(r-2j-\ve-s_j)\fr}{r-2j-\ve}+\frac{jf(2j+\ve)}{r-2j-\ve}.\label{h2}
\end{eqnarray}
We further require
\begin{equation}\label{rl}
	r>2j+\ve+s_j.
\end{equation}
Under this choice for $r$, we can combine \eqref{h1} and \eqref{h2} to deduce
\begin{eqnarray}
	\lefteqn{-j\fp+\fr+\frac{(2j+\ve-q)(j-\mq)\fr}{(q-\lz)(r-2j-\ve)}}\nonumber\\
	&\leq&\frac{(r-2j-\ve-s_j)}{r-2j-\ve}\left[1-\frac{j-\mq}{q-\lz}-\frac{(1-\eta)(\ve+s_j)}{2j+\ve-\lz}\right]e^{\frac{r-2j-\ve}{s_j}}f(2j+\ve)\nonumber\\
	&&+\frac{(r-2j-\ve-s_j)}{r-2j-\ve}\left[\frac{j-\mq}{q-\lz}+\frac{(1-\eta)(r-2j+s_j)}{2j+\ve-\lz}\right]f(2j+\ve)+\frac{jf(2j+\ve)}{r-2j-\ve}.\nonumber
\end{eqnarray}
Utilizing this in \eqref{ja2}, we arrive at
\begin{eqnarray}
	\|w\|_{\infty,Q_T}&\leq&	16\|w(\cdot,0)\|_{\infty,\rn}+ c\|w\|_{\lz,Q_T}^{s_1}\|w\|_{2j+\ve,Q_T}^{\beta_2},\label{est12}
\end{eqnarray}
where
\begin{eqnarray}
	\beta_2&=&-\frac{(j-\mq)(2j+\ve)(r-q)}{(q-\lz)(r-2j)(r-2j-\ve)}+\frac{j(2j+\ve)}{(r-2j-\ve)(r-2j)}\nonumber\\
	&&+\frac{(r-2j-\ve-s_j)e^{\frac{r-2j-\ve}{s_j}}(2j+\ve)}{(r-2j-\ve)(r-2j)}\left[1-\frac{j-\mq}{q-\lz}-\frac{(1-\eta)(\ve+s_j)}{2j+\ve-\lz}\right]\nonumber\\
	&&+\frac{(r-2j-\ve-s_j)(2j+\ve)}{(r-2j-\ve)(r-2j)}\left[\frac{j-\mq}{q-\lz}+\frac{(1-\eta)(r-2j+s_j)}{2j+\ve-\lz}\right].\label{b2}
\end{eqnarray}
We must have
\begin{equation}\label{dh6}
	\beta_2<\frac{2j+\ve}{2j+\ve-\lz}.
\end{equation}
Plug \eqref{b2} into this and simplify the resulting inequality to obtain
\begin{eqnarray}
	\lefteqn{-\frac{(j-\mq)(r-\lz)}{(q-\lz)}+r-\mq}\nonumber\\
	&&+(r-2j-\ve-s_j)e^{\frac{r-2j-\ve}{s_j}}\left[1-\frac{j-\mq}{q-\lz}-\frac{(1-\eta)(\ve+s_j)}{2j+\ve-\lz}\right]\nonumber\\
	&&+(r-2j-\ve-s_j)\left[\frac{j-\mq}{q-\lz}+\frac{(1-\eta)(r-2j+s_j)}{2j+\ve-\lz}\right]<\frac{(r-2j)(r-\lz)}{2j+\ve-\lz},\nonumber
\end{eqnarray}
from whence it follows
\begin{eqnarray}
\lefteqn{\left[(\ve+s_j)e^{\frac{r-2j-\ve}{s_j}}-(r-2j+s_j)\right]\frac{(r-2j-\ve-s_j)\eta}{2j+\ve-\lz}}	\nonumber\\
&<&\frac{(r-2j)(r-\lz)}{2j+\ve-\lz}+\frac{(j-\mq)(r-\lz)}{(q-\lz)}-r+\mq\nonumber\\
&&+(r-2j-\ve-s_j)\left[e^{\frac{r-2j-\ve}{s_j}}\left[\frac{j-\mq}{q-\lz}+\frac{(\ve+s_j)}{2j+\ve-\lz}-1\right]-\frac{j-\mq}{q-\lz}-\frac{(r-2j+s_j)}{2j+\ve-\lz}\right]\nonumber\\
&\equiv& h(r).\label{dh3}
\end{eqnarray}
We easily ckeck that
\begin{equation}
	\frac{(\ve+s_j)e^{\frac{r-2j-\ve}{s_j}}}{2j+\ve-\lz}-\frac{(r-2j+s_j)}{2j+\ve-\lz}>0\ \ \mbox{for $r>2j+\ve$.}\nonumber
\end{equation}
Indeed, there hold
\begin{eqnarray}
	\left.(\ve+s_j)e^{\frac{r-2j-\ve}{s_j}}-(r-2j+s_j)	\right|_{r=2j+\ve}&=&0,\nonumber\\
	\left[(\ve+s_j)e^{\frac{r-2j-\ve}{s_j}}-(r-2j+s_j)\right]^\prime&>&0\ \ \mbox{for $r>2j+\ve$.}\nonumber
\end{eqnarray}
Hence, we must show that there exists a $r>2j+\ve+s_j$ such that
\begin{eqnarray}
h(r)>0 .\label{dh4}
\end{eqnarray}
Note from \eqref{sj} that
\begin{eqnarray}
	\frac{r-2j+s_j}{2j+\ve-\lz}&=&\frac{r-j-\frac{(2j+\ve-q)(j-\mq)}{q-\lz}}{2j+\ve-\lz}\nonumber\\
	&=&\frac{r-\mq-\frac{(2j+\ve-\lz)(j-\mq)}{q-\lz}}{2j+\ve-\lz}\nonumber\\
	&=&\frac{r-\mq}{2j+\ve-\lz}-\frac{(j-\mq)}{q-\lz}.\nonumber
\end{eqnarray}
By the same token,
\begin{eqnarray}
	\frac{2j-s_j-\mq}{2j+\ve-\lz}&=&\frac{j+\frac{(2j+\ve-q)(j-\mq)}{q-\lz}-\mq}{2j+\ve-\lz}\nonumber\\
	&=&\frac{(j-\mq)}{q-\lz},\nonumber\\
	\frac{s_j+\ve}{2j+\ve-\lz}&=&\frac{j-\frac{(2j+\ve-q)(j-\mq)}{q-\lz}+\ve}{2j+\ve-\lz}\nonumber\\
	&=&\frac{2j-\mq-\frac{(2j+\ve-\lz)(j-\mq)}{q-\lz}+\ve}{2j+\ve-\lz}\nonumber\\
	&=&\frac{2j-\mq+\ve}{2j+\ve-\lz}-\frac{j-\mq}{q-\lz}.\nonumber
\end{eqnarray}
Incorporate the preceding three equations into the expression for $h(r)$ in \eqref{dh3} to derive
\begin{eqnarray}
h(r)&=&\frac{(r-2j)(r-\lz)}{2j+\ve-\lz}+\frac{(j-\mq)(r-\lz)}{(q-\lz)}-r+\mq\nonumber\\
&&+(r-2j-\ve-s_j)\left[e^{\frac{r-2j-\ve}{s_j}}\left[\frac{2j-\mq+\ve}{2j+\ve-\lz}-1\right]-\frac{r-\mq}{2j+\ve-\lz}\right]\nonumber\\
&=&\frac{(r-2j)(r-\lz)}{2j+\ve-\lz}+\frac{(j-\mq)(r-\lz)}{(q-\lz)}-\frac{(r-\mq)(r-s_j-\lz)}{2j+\ve-\lz}\nonumber\\
&&-\frac{(r-2j-\ve-s_j)(\mq-\lz)e^{\frac{r-2j-\ve}{s_j}}}{2j+\ve-\lz}\nonumber\\
&=&\frac{-(2j-s_j-\mq)r+\lz(2j-\mq)}{2j+\ve-\lz}+\frac{(j-\mq)(r-\lz)}{(q-\lz)}\nonumber\\
	&&-\frac{(r-2j-\ve-s_j)(\mq-\lz)e^{\frac{r-2j-\ve}{s_j}}}{2j+\ve-\lz}\nonumber\\
	&=&\frac{\lz(2j-\mq)}{2j+\ve-\lz}-\frac{\lz(j-\mq)}{(q-\lz)}-\frac{(r-2j-\ve-s_j)(\mq-\lz)e^{\frac{r-2j-\ve}{s_j}}}{2j+\ve-\lz}.\nonumber
\end{eqnarray}
In view of \eqref{ml}, for \eqref{dh4} to hold, we must have
\begin{equation}
	\frac{(2j-\mq)}{2j+\ve-\lz}-\frac{(j-\mq)}{(q-\lz)}>0.\nonumber
\end{equation}
Solve this for $\ve$ to deduce
\begin{eqnarray}
	\ve&<&\frac{q-\lz}{j-\mq}\left[2j-\mq-\frac{(j-\mq)(2j-\lz)}{q-\lz}\right]\nonumber\\
	&=&\frac{q-\lz}{j-\mq}\left[j-\frac{(j-\mq)(2j-q)}{q-\lz}\right],\nonumber
\end{eqnarray}
which is exactly our assumption \eqref{eu}.

In summary, the order in which we choose our parameters is as follows: Let $q$ be given as in \eqref{qcon2}. By virtue of  \eqref{j2}, we may take
\begin{equation}
	j\in\left(\max\left\{\frac{q}{2},j_1\right\}, j_2\right).\nonumber
\end{equation}
This implies \eqref{jl} and  enables us to select $\ve$ as in \eqref{eu},  which, in turn,  guarantees \eqref{dh4} for  $r$ close to $2j+\ve+s_j$ from the right-hand side. That is, we have both \eqref{rl} and \eqref{dh4}. Equipped with this, we can choose $\eta\in(0,1)$ so that \eqref{dh3} is satisfied.

Without any loss of generality, we may assume
\begin{equation}\label{h4}
	\|w\|_{2j+\ve,Q_T}>1. 
\end{equation}
Otherwise, \eqref{ha7} would be enough to imply our theorem. Under \eqref{h4}, we may suppose that the last exponent $\beta_2$ in \eqref{est12} is positive.
Were this not true, \eqref{h4} combined with \eqref{est12} would yield our theorem. 
In view of \eqref{qcon} and \eqref{lz}, we can form the interpolation inequality
\begin{equation}
	\|w\|_{2j+\ve,Q_T}\leq \|w\|_{\infty,Q_T}^{\frac{2j+\ve-\lz}{2j+\ve}}\|w\|_{\lz,Q_T}^{\frac{\lz}{2j+\ve}}.\nonumber
\end{equation}
Collect this  in \eqref{est12} and keep in mind that $\beta_2>0$ to get
\begin{eqnarray}
	\|w\|_{\infty,Q_T}&\leq& 16\|w(\cdot,0)\|_{\infty,\rn}+c\|w\|_{\lz,Q_T}^{s_1+\frac{\lz\beta_2}{2j+\ve}}\|w\|_{\infty,Q_T}^{\frac{(2j+\ve-\lz)\beta_2}{2j+\ve}},\nonumber
\end{eqnarray}
According to \eqref{dh6}, we have
\begin{equation}
	\frac{(2j+\ve-\lz)\beta_2}{2j+\ve}\in (0,1).\nonumber
\end{equation}
Then \eqref{es} follows from a suitable application of Young's inequality \eqref{you}.
The proof of Theorem \ref{thm} is completed.
\end{proof}


\end{document}